\documentclass[a4,12pt]{amsart}
\usepackage{amssymb,amstext,amsmath,amscd,amsthm,amsfonts,enumerate,latexsym, comment}
\usepackage{color}
\usepackage[dvipdfmx]{graphicx}

\usepackage[all]{xy}

\theoremstyle{plain}
\newtheorem{thm}{Theorem}[section]

\newtheorem*{thm*}{Theorem}
\newtheorem*{cor*}{Corollary}

\newtheorem{prop}[thm]{Proposition}

\newtheorem{lem}[thm]{Lemma}
\newtheorem{cor}[thm]{Corollary}

\newtheorem{claim}{Claim}
\newtheorem*{claim*}{Claim}

\theoremstyle{definition}
\newtheorem{defn}[thm]{Definition}

\newtheorem{ex}[thm]{Example}

\newtheorem{rem}[thm]{Remark}

\theoremstyle{remark}

\numberwithin{equation}{thm}

\def\Min{\operatorname{Min}}

\def\Max{\operatorname{Max}}

\def\m{\mathfrak m}

\newcommand{\rme}{\mathrm{e}}

\newcommand{\rmQ}{\mathrm{Q}}

\newcommand{\calF}{\mathcal{F}}

\newcommand{\calX}{\mathcal{X}}

\newcommand{\jcb}{\operatorname{J}}

\newcommand{\mapright}[1]{%
\smash{\mathop{%
\hbox to 1cm{\rightarrowfill}}\limits^{#1}}}

\newcommand{\mapleft}[1]{%
\smash{\mathop{%
\hbox to 1cm{\leftarrowfill}}\limits_{#1}}}

\def\Ass{\operatorname{Ass}}

\title[Construction and finite generation of strict closure of rings]{Construction and finite generation of strict closure of rings}

\author[R. Isobe]{Ryotaro Isobe}

\address{General Education Division, National Institute of Technology, Oshima College, 1091-1, Oazakomatsu, Suooshima-cho, Oshima-gun, Yamaguchi, 742-2193, Japan}
\email{r.isobe.math@gmail.com}

\thanks{The author was supported by JSPS KAKENHI Grant Number 21K13767.}

\begin{document}

\maketitle

\setlength{\baselineskip} {15.2pt}

\begin{abstract}
This paper studies the construction and finite generation of strict closure of rings. We determine their structure when $R$ is a Cohen-Macaulay semi-local ring of dimension one with $\dim R_M=1$ for every $M\in\Max R$. By using this, a characterization of finite generation of the strict closure is given. 

\end{abstract}

%{\footnotesize \tableofcontents}

%%%%%%%%%%%%%%%%%%%%%%%%%%%%%%%%%%%%%%%%%%%%%%%%%%%%%%%%%%%%%%%%%%%%%%%%%%%%%%%%%%%%%%%%%%%%%%%%%%%%%%%%%%%%%%%%%%%%%%%%%%%%%%%%%%%%%%%%%%%%%%%%

\section{Introduction}\label{intro}

In this paper, we investigate the structure of strict closure $R^*$ of a commutative ring $R$, which is an intermediate ring between $R$ and its integral closure $\overline{R}$ consisting of those elements $\alpha$ in $\overline{R}$ such that $\alpha\otimes 1=1\otimes \alpha$ in $\overline{R}\otimes_R \overline{R}$.
The notion of strict closure of rings was introduced by J. Lipman \cite{L} in 1971, and he established the underlying theory.
 Since $[R^*]^*=R^*$ and $R^*\subseteq T^*$ if $T$ is an intermediate ring between $R$ and $\overline{R}$, $[-]^*$ is actually a closure operation.
 In the same paper, J. Lipman also introduced the notion of Arf rings for Cohen-Macaulay semi-local rings $R$ satisfying that every localization $R_M$ at a maximal ideal $M$ is dimension one, which is strongly related to strict closedness.   
  Arf rings were originally studied in the classification of certain singular points of plane curves by C. Arf \cite{Arf}, and Lipman generalized them by extracting the essence of the rings.
  A typical example of an Arf ring is a Cohen-Macaulay local ring of dimension one with multiplicity at most two {\cite[Example, page 664]{L}}, and semi-normal Cohen-Macaulay local ring of dimension one.
As is mentioned in \cite{L}, it was conjectured that $R$ is an Arf ring if and only if $R=R^*$. This conjecture was solved in \cite{L} when $R$ contains a field ({\cite[Proposition 4.5, Theorem 4.6]{L}}), and was settled in \cite{C} for its general case ({\cite[Theorem 4.4]{C}}).

The first purpose of this paper is to clarify the structure of $R^*$ and to give specific way of constructing $R^*$, when $R$ is not an Arf ring.
In \cite{Arf}, C. Arf gave a method for construction of {\it the Arf closure} of $R$, which is the smallest Arf ring between $R$ and $\overline{R}$, when $R$ is a subring of the formal power-series ring $k[[t]]$ over a field $k$.   
In \cite{L}, J. Lipman gave another method for its construction when $R$ is a Cohen-Macaulay semi-local rings satisfying that every localization $R_M$ at a maximal ideal $M$ is dimension one, and prove that the Arf closure coincides with $R^*$ provided $\overline{R}$ is a finitely generated $R$-module. 
In contrast, N. Endo and S. Goto [3] gave a practical method of construction of strictly closed rings, and found many such rings; for example, the Stanley-Reisner rings of simplicial complexes and F-pure rings satisfying $(S_2)$. 
Moreover, in \cite{EGI}, N. Endo, S. Goto, and the author studied some topics on strict closure, and computed concrete examples for one-dimensional case as well as higher dimensional case, giving an upper bound for $R^*$. 

However, even for the one-dimensional case, there is still no result that clearly describe the structure of $R^*$ when $\overline{R}$ is not a finitely generated $R$-module.
Therefore, in the current paper, we determine their structure when $R$ is a Cohen-Macaulay semi-local ring of dimension one with $\dim R_M=1$ for every $M\in\Max R$, and compute some examples of $R^*$ for the case where $\overline{R}$ is finitely generated as well as not finitely generated.

Let  $R$ be a Cohen-Macaulay semi-local ring of dimension one with $\dim R_M=1$ for every $M\in\Max R$, and let $\operatorname{J}(R)$ denote the Jacobson radical of $R$. 
We set $R_1=R^{\operatorname{J}(R)}=\bigcup_{n\ge0}[{\operatorname{J}(R)}^n:{\operatorname{J}(R)}^n]$, {\it the blow-up of $R$ at $\operatorname{J}(R)$}, and define recursively 
\begin{center}
$R_n=$
$
\begin{cases}
\  R & \ \ (n=0)\\
\ [R_{n-1}]_1 & \ \ (n>0)\\
\end{cases}
$
\end{center}
for each $n\ge 0$. In contrast, 
we set $\m_n=\operatorname{J}(R_n)$ and
\begin{center}
$R_{(n)}=$
$
\begin{cases}
\  R & \ \ (n=0)\\
\ R+\m_0R_1+\cdots + \m_0\m_1\cdots \m_{n-1}R_n & \ \ (n>0)\\
\end{cases}
$
\end{center}
for each $n\ge 0$.
The first main result of this paper is the following.

\begin{thm}[Theorem \ref{main}]
$R^*=\underset{n\ge 0}{\bigcup}R_{(n)}$.
\end{thm}
This theorem enables us to compute the strict closure $R^*$ when $\overline{R}$ is not necessarily a finitely generated $R$-module. 

The second purpose of this paper is to characterize the finite generation of $R^*$. It is well known that the integral closure $\overline{R}$ is a finitely generated $R$-module if and only if $R$ is analytically unramified, when $R$ is a Cohen-Macaulay semi-local ring of dimension one with $\dim R_M=1$ for every $M\in\Max R$. By using Theorem 1.1, we prove the strict closure version of  that one, which reveals the similarity between $R^*$ and $\overline{R}$.

\begin{thm}[Theorem \ref{4.2}]
Suppose that $R$ is a Cohen-Macaulay semi-local ring of dimension one with $\dim R_M=1$ for every $M\in\Max R$. Let $\widehat{R}$ denote the $\jcb(R)$-adic completion of $R$.
Then, the following conditions are equivalent. 
\begin{enumerate}[$(1)$]
\item
$R^*$ is a finitely generated $R$-module.

\item
$(\sqrt{(0)})^2=(0)$ in $\widehat{R}$.

\item
$(P\widehat{R}_P)^2=(0)$ in $\widehat{R}_P$ for every $P\in \Min \widehat{R}$.

\end{enumerate}

\end{thm}

We now explain how this paper is organized. In Section \ref{sec2}, we will summarize definitions of strict closure of rings and Arf rings, and some results which we subsequently need. In Section \ref{sec3}, we will give the proof of Theorem 1.1 (Theorem \ref{main}). By using this result, we compute concrete examples of $R^*$ for the case where $\overline{R}$ is finitely generated as well as not finitely generated (Example \ref{3.10}, \ref{3.11}). The proof of Theorem 1.2 (Theorem \ref{4.2}) will be given in Section \ref{sec4}.

Throughout this paper,  let $R$ be a commutative ring, let $W(R)$ be the set of non-zerodivisors on $R$, and let $\calF_R$ be the set of {\it open} ideals in $R$, that is, the ideals of $R$ that contain a non-zerodivisor on $R$. For an ideal $I$ in $R$, $\overline{I}$ denotes the integral closure of $I$ in $R$, and set $\calX_R=\{ I\in \calF_R \mid \overline{I}=I \}$. We set $X:Y=\{a\in\rmQ(R) \mid aY\subseteq X\}$ for $R$-submodules $X$ and $Y$ of the total ring of fractions $\rmQ(R)$.

%%%%%%%%%%%%%%%%%%%%%%%%%%%%%%%%%%%%%%%%%%%%%%%%%%%%%%%%%%%%%%%%%%%%%%%%%%%%%%%%%%%%%%%%%%%%%%%%%%%%%%%%%%%%%%%%%%%%%%%%%%%%%%%%%%%%%%%%%%%%%%%%
\section{Strict closure of rings and Arf rings}\label{sec2}

In this section, we summarize definitions of strict closure of rings and Arf rings, and some results.
First, we introduce the definition of strict closure of rings. Let $S/R$ be an extension of commutative rings, and we set
$$
R_S^{*}=\left\{ \alpha \in S \mid \alpha \otimes 1 = 1 \otimes \alpha \text{ in } S \otimes_RS\right\}
$$
which is a kernel of an $R$-linear map 
$$
\sigma: S \to S \otimes_RS, \ \  \alpha \mapsto \alpha \otimes 1 - 1 \otimes \alpha.
$$
Then $R^*_S$ is an intermediate ring between $R$ and $S$, which is called {\it the strict closure} of $R$ in $S$. In particular, we denote by $R^*$, if $S=\overline{R}$, where $\overline{R}$ denotes the integral closure of $R$ in its total ring $\rmQ(R)$ of fractions. Since $[R^*]^*=R^*$ and $R^*\subseteq T^*$ if $T$ is an intermediate ring between $R$ and $\overline{R}$, $[-]^*$ is actually a closure operation. In contrast, since $R^*_S$ is a kernel of an $R$-linear map $\sigma$, the notion of strict closure is compatible with flat base changes. The reader may consult with \cite{L, C, EG, EGI} about properties of strict closures.   

The following proposition gives an upper bound for strict closures of rings, and we will use later in Section 3.

\begin{prop}[{\cite[Proposition 2.2]{EGI}}]\label{3.6}
Let $S/R$ be an extension of commutative rings. Then 
$$
R^*_S \subseteq R + MS \ \ \text{in} \ \ S
$$
for every $M \in \Max R$.

\end{prop}

Hence we have the following corollary.

\begin{cor}\label{2.5}
Suppose that $(R, \m)$ is a local ring. If $\m \overline{R}\subseteq R$, then $R=R^*$.
\end{cor}

\vspace{1em}

Second, we introduce the definition of Arf rings and some results based on \cite{L}, and explain the relationship between Arf property and strictly closedness. For an arbitrary commutative ring $R$, let $W(R)$ be the set of non-zerodivisors on $R$. We denote by $\calF_R$ the set of ideals in $R$ that contain a non-zerodivisor on $R$.

Let $I\in \calF_R$. We consider the tower of $R$-algebras as follows:
$$R\subseteq I:I\subseteq I^2:I^2\subseteq \cdots \subseteq I^n:I^n \subseteq \cdots \subseteq \rmQ(R), \ \text{where}\  I^n:I^n=\{a\in\rmQ(R) \mid aI^n\subseteq I^n\}.$$
We set $R^I=\bigcup_{n\ge0}[I^n:I^n]$. Then, $R^I$ is an intermediate ring between $R$ and $\rmQ(R)$. If $a\in I$ is a reduction of $I$, that is, $I^{r+1}=aI^r$ for some $r\ge 0$, then $R^I=I^n:I^n$ for any $n\ge r$ and we have 
$$R^I=R\left[ \tfrac{I}{a}  \right] = \frac{I^r}{a^r},\ \text{where} \  \frac{I}{a}= \{ \frac{x}{a} \mid x\in I \}\subseteq \rmQ(R).$$
In particular, the following holds true.

\begin{lem}[c.f. {\cite[Lemma 1.11.]{L}}]\label{2.1}
Let $I\in\calF_R$. Suppose that there exists $a\in I$ and $r\ge0$ such that $I^{r+1}=aI^r$. Then, the following conditions are equivalent.
\begin{enumerate}[$(1)$]
\item
$R^I=I:I.$
\item
$I:I=\frac{I}{a}.$
\item
$I^2=aI.$
\end{enumerate}
Therefore, condition $(3)$ is independent of the choice of reductions $a\in I$. An ideal $I\in \calF_R$ that satisfies condition $(3)$ is called a stable ideal.
\end{lem}

In what follows, unless otherwise specified, we assume that $R$  is a Cohen-Macaulay semi-local ring of dimension one with $\dim R_M=1$ for every $M\in\Max R$.
Under this assumption, the notion of Arf rings is defined as follows. We set $\calX_R=\{I\in\calF_R \mid \overline{I}=I \}$.

\begin{defn}[\cite{L}]\label{2.2}
$R$ is called an {\it Arf ring} if the following conditions hold:
\begin{enumerate}[$(1)$]
\item
For every $I\in\calX_R$, there exists an element $a\in I$ such that $I=\overline{aR}$, that is, $I^{n+1}=aI^n$ for some $n\ge0$.  
\item
If $x, y, z \in R$ such that $x\in W(R)$ and $\frac{y}{x}, \frac{z}{x}\in \overline{R}$, then $\frac{yz}{x}\in R$.
\end{enumerate}
\end{defn}

Arf rings are characterized in terms of the stability of the integrally closed ideals in $\calX_R$.

\begin{thm} [{\cite[Theorem 2.2]{L}}]\label{2.3}
The following conditions are equivalent.
\begin{enumerate}[$(1)$]
\item
$R$ is an Arf ring.
\item
For every $I\in \calX_R$, $I$ is stable.

\end{enumerate}
\end{thm}

For a Noetherian local ring $R$,  $v(R)$ denotes the embedding dimension of $R$, and $\rme(R)$ denotes the multiplicity of $R$ with respect to the maximal ideal. When $R$ is an Arf ring, the equation $v(R_M)=\rme(R_M)$ holds for every $M\in\Max R$ (i.e., $R_M$ has minimal multiplicity), because $M^2=aM$ for some $a\in M$. In contrast, even if $R_M$ has minimal multiplicity for every $M\in \Max R$, $R$ is not necessarily an Arf ring.

The strictly closedness is strongly related to the Arf property as follows. The following theorem was proved in \cite{L} when $R$ contains a field ({\cite[Proposition 4.5, Theorem 4.6]{L}}), and was settled in \cite{C} for its general case ({\cite[Theorem 4.4]{C}}).

\begin{thm}[\cite{L, C}]\label{2.4}
The following conditions are equivalent.
\begin{enumerate}[$(1)$]
\item
$R=R^*$ in $\overline{R}$.
\item
$R$ is an Arf ring.
\end{enumerate}

\end{thm}

In Section \ref{sec3}, we consider the construction of $R^*$ when $R$ is not necessarily an Arf ring.

\section{Construction of $R^*$}\label{sec3}

In this section, we explore the construction of $R^*$. Let $R$ be an arbitrary commutative ring. We take a tower of rings $R=R_0\subseteq R_1 \subseteq \cdots \subseteq R_n \subseteq \cdots \subseteq \overline{R}$  such that $\underset{n\ge 0}{\bigcup}R_{n}=\overline{R}$. We begin with the following theorem, which plays a key role in this section. The idea for the proof of this theorem has already been written in \cite[Proof of Theorem 4.6.]{L}.

\begin{thm}\label{3.1}
$R^*=\underset{n\ge 0}{\bigcup}R_{R_n}^*$.

\end{thm}

\begin{proof}
Because $R^*_{R_n}$ is a kernel of an $R$-linear map  $\sigma : R_n \to R_n \otimes_R R_n$ defined by $\sigma (\alpha) = \alpha \otimes 1 - 1 \otimes \alpha$, we have the exact sequence 
$$0 \to R^*_{R_n} \to R_n \overset{\sigma} \to R_n \otimes_R R_n$$
of $R$-modules for every $n\ge 0$. By taking the direct limit to these exact sequences, we get the commutative diagram 
$$
\xymatrix{
0 \ar[r] &  \underset{n\to \infty}{\lim}R^*_{R_n} \ar[r]^{}\ar@{=}[d] & \underset{n\to \infty}{\lim}R_n \ar[r]\ar@{=}[d]^{} & \underset{n\to \infty}{\lim}(R_n\otimes_R R_n)  \ar[d]^{\wr}  \\
0 \ar[r] & \underset{n\ge 0}{\bigcup}R^*_{R_n} \ar[r]^{}  & \overline{R} \ar[r]^{\sigma} & \overline{R}\otimes_R \overline{R} 
}
$$
of $R$-modules, since $\underset{n\to \infty}{\lim} R_n= \underset{n\ge 0}{\bigcup}R_{n}=\overline{R}$. This implies that $R^*=\underset{n\ge 0}{\bigcup}R_{R_n}^*$.
\end{proof}

Therefore, in order to compute $R^*$, we have only to compute $R^*_{R_n}$ for some tower $R=R_0\subseteq R_1 \subseteq \cdots \subseteq R_n \subseteq \cdots \subseteq \overline{R}$ of rings which satisfies $\underset{n\ge 0}{\bigcup}R_{n}=\overline{R}$.

From now on, let  $R$ be a Cohen-Macaulay semi-local ring of dimension one with $\dim R_M=1$ for every $M\in\Max R$, and we denote by $\operatorname{J}(R)$ as the Jacobson radical of $R$. 
We set $R_1=R^{\operatorname{J}(R)}$ and define recursively 
\begin{center}
$R_n=$
$
\begin{cases}
\  R & \ \ (n=0)\\
\ [R_{n-1}]_1 & \ \ (n>0)\\
\end{cases}
$
\end{center}
for each $n\ge 0$. Then, we get the tower
$$R=R_0\subseteq R_1\subseteq \cdots \subseteq R_n \subseteq \cdots \subseteq \overline{R}$$
of rings, and it is known that $\underset{n\ge 0}{\bigcup}R_{n}=\overline{R}$ holds (\cite[Proof of Theorem 4.6.]{L}).

In contrast, 
we set $\m_n=\operatorname{J}(R_n)$ and
\begin{center}
$R_{(n)}=$
$
\begin{cases}
\  R & \ \ (n=0)\\
\ R+\m_0R_1+\cdots + \m_0\m_1\cdots \m_{n-1}R_n & \ \ (n>0)\\
\end{cases}
$
\end{center}

for each $n\ge 0$.
Because $\m_n$ is an open ideal in $R_n$, there exists an element $\alpha_n\in R_{n+1}$ such that $\m_n R_{n+1}= \alpha_n R_{n+1}$ for each $n\ge 0$ (\cite[Proposition 1.1.]{L}), and we can write
\begin{eqnarray*}
R_{(n)} &=& R+\m_0R_1+\cdots + \m_0\m_1\cdots \m_{n-1}R_n\\
&=& R+\alpha_0R_1+\cdots + \alpha_0\m_1\cdots \m_{n-1}R_n\\
&=& R+\alpha_0(R_1+\cdots + \m_1\cdots \m_{n-1}R_n)\\
&=& R+\alpha_0\cdot [R_1]_{(n-1)}\ \ (n\ge 1).
\end{eqnarray*}

The purpose of this section is to prove the following, which is the main result of this paper. 

\begin{thm}\label{main}

$R^*_{R_n}=R_{(n)}$ for every $n\ge 0$. Therefore, we have $R^*=\underset{n\ge 0}{\bigcup}R_{(n)}$.

\end{thm}

In order to prove Theorem \ref{main}, we need some preliminaries. 
We begin with the following. 

\begin{lem}\label{3.3}
For each $n\ge 1$, the following assertions hold true.

\begin{enumerate}[$(1)$] 
\item
$\Max R_{(n)}=\{ M+ \alpha_0\cdot [R_1]_{(n-1)} \mid M\in \Max R \}$.
\item
$\operatorname{J}(R_{(n)})=\alpha_0\cdot [R_1]_{(n-1)}$.

\end{enumerate}

\end{lem}

\begin{proof}
$(1)$
Let $N\in \Max R_{(n)}$. Because $R\subseteq R_{(n)} \subseteq \overline{R}$, there exists a maximal ideal $L\in \Max \overline{R}$ such that $L\cap R_{(n)} =N$. Hence we have 
\begin{eqnarray*}
N=L\cap R_{(n)} &=& L\cap (R+\alpha_0\cdot [R_1]_{(n-1)})\\
&=& (L\cap R)+ \alpha_0\cdot [R_1]_{(n-1)}
\end{eqnarray*}
and $L\cap R\in \Max R$, where the third equality follows from the fact that $\alpha_0\cdot [R_1]_{(n-1)} \subseteq \m_0 \overline{R} \subseteq L$. 

Conversely, let $M\in \Max R$ and set $N=M+ \alpha_0\cdot [R_1]_{(n-1)}$. Then, $N$ is an ideal of $R_{(n)}$ and the natural map 
$$R/M \to R_{(n)}/N=R+\alpha_0\cdot [R_1]_{(n-1)}/M+\alpha_0\cdot [R_1]_{(n-1)}$$
is surjective, which implies that $R/M\cong R_{(n)}/N$ as rings and $N\in \Max R_{(n)}$.

$(2)$ 
We have $\operatorname{J}(R_{(n)})=\underset{M\in \Max R}{\bigcap} \left( M+\alpha_0\cdot [R_1]_{(n-1)} \right)\supseteq \alpha_0\cdot [R_1]_{(n-1)}$ by assertion $(1)$. Moreover, for each $M\in \Max R$, we have
\begin{eqnarray*}
\left[ \operatorname{J}(R_{(n)}) \right]_M &=& \left[ M+\alpha_0\cdot [R_1]_{(n-1)} \right]_M\\
&=& \left[ M+\m_0\cdot [R_1]_{(n-1)} \right]_M\\
&=& MR_M+ M\left[ [R_1]_{(n-1)}   \right]_M\\
&=& M\left[ [R_1]_{(n-1)}   \right]_M\\
&=& \left[\alpha_0\cdot [R_1]_{(n-1)} \right]_M,
\end{eqnarray*}

where the second equality follows from $\alpha_0 R_1=\m_0  R_1$. This concludes that $\operatorname{J}(R_{(n)})=\alpha_0\cdot [R_1]_{(n-1)}$, as desired. 
\end{proof}

\begin{rem}\label{rem}
If $R$ is a local ring, then so is $R_{(n)}$ and these residue class fields are isomorphic. Moreover, because $\operatorname{J}(R_{(n)})^2=\alpha_0 \operatorname{J}(R_{(n)})$, $R_{(n)}$ has minimal multiplicity. 
\end{rem}

We then have the following.

\begin{prop}\label{3.4}
$\m_0\cdots \m_n$ is an integrally closed ideal in $R_{(n)}$ for each $n\ge 0$.

\end{prop}

\begin{proof}
We prove by induction on $n$. If $n=0$, the assertion is obvious. We may assume that $n>0$ and our assertion holds true for $n-1$. Because $\operatorname{J}(R_{(n)})=\alpha_0\cdot [R_1]_{(n-1)}$, we have $\operatorname{J}(R_{(n)})^2=\alpha_0\operatorname{J}(R_{(n)})$, so that 
$$[R_{(n)}]_1=[R_{(n)}]^{\operatorname{J}(R_{(n)})}=\alpha_0\cdot [R_1]_{(n-1)}:\alpha_0\cdot [R_1]_{(n-1)}=[R_1]_{(n-1)}$$ 
by Lemma \ref{2.1}.  In contrast, by the hypothesis of induction, $\m_1\cdots \m_n$ is an integrally closed ideal in $[R_1]_{(n-1)}$. Therefore, $\alpha_0\cdot \m_1\cdots \m_n=\m_0\m_1\cdots \m_n$ is an integrally closed ideal in $R_{(n)}$ by \cite[Lemma 2.4]{I}, as desired. 
\end{proof}

Here, we note the following lemma, which we need to prove Proposition \ref{3.7}.

\begin{lem}[{\cite[Lemma 4.3]{C}}]\label{3.5}
Let $(R, \m)$ be a one-dimensional Cohen-Macaulay local ring. Suppose that $\m^2 = \alpha\m$ for some $\alpha\in \m$. Set $R_1 = R^{\m} = \frac{\m}{\alpha}$.  Let $R_1 \subseteq A \subseteq \overline{R}$ be an intermediate ring such that $A$ is a finitely generated $R$-module, let $\tau : A\otimes_R A \to A \otimes_{R_1} A$ be an $R$-algebra map such that $\tau(x \otimes y) = x \otimes y$ for every $x, y \in A$. Then the sequence 
$$0\to (0):_{A \otimes_RA} \alpha \to A\otimes_R A\overset{\tau}{\to} A\otimes_{R_1} A \to 0$$
is exact.
\end{lem}

We then have the following.

\begin{prop}\label{3.7}
$\left[ R_{(n)}\right]^*_{R_n}=R_{(n)}$ for every $n\ge 0$.
\end{prop}

\begin{proof}
We prove by induction on $n$. The case $n=0$ is obvious. Suppose  that $n>0$ and our assertion holds true for $n-1$. Since $R_{(n)} \subseteq [ R_{(n)} ]^*_{R_n}$, passing to the localization at a maximal ideal of $R$, we may assume  that both $R$ and $R_{(n)}$ are local (notice that $R_{(n)}$ commutes with the localization at a maximal ideal of $R$). Because $\alpha_0\cdot [R_1]_{(n-1)}$ is the maximal ideal of $R_{(n)}$, we have 
$$[ R_{(n)} ]^*_{R_n} \subseteq R_{(n)} + \alpha_0\cdot [R_1]_{(n-1)} \cdot R_n =R_{(n)}+\alpha_0\cdot R_n $$
by Proposition \ref{3.6}. Take $x\in [ R_{(n)} ]^*_{R_n}$ and write $x=x_1+\alpha_0 y$  $(x_1\in R_{(n)}, y\in R_n )$. Since $\alpha_0 y=x-x_1\in [ R_{(n)} ]^*_{R_n}$, we get
$$\alpha_0 (y\otimes 1-1\otimes y) = \alpha_0 y\otimes 1- 1\otimes \alpha_0 y = 0$$ 
in $R_n \otimes_{R_{(n)}} R_n$, so that 
$$y\otimes 1- 1\otimes y \in (0):_{R_n\otimes_{R_{(n)}} R_n} \alpha_0.$$
 Therefore, by using Lemma \ref{3.5}, we obtain $y\otimes 1-1\otimes y=0$ in $R_n\otimes_{[ R_1 ]_{(n-1)}} R_n $ since $[R_{(n)}]_1=\left[ R_1 \right]_{(n-1)} $, which implies $y\in \left[ [R_1]_{(n-1)}  \right]^*_{R_n} $. 
Let $B=R_1$. Then $\left[ [R_1]_{(n-1)}  \right]^*_{R_n} =\left[ B_{(n-1)}   \right]^*_{B_{n-1}}$. By the hypothesis of induction, we have 
$$\left[ B_{(n-1)}   \right]^*_{B_{n-1}}=B_{(n-1)}=\left[ R_1 \right]_{(n-1)}, $$
 so that $y \in \left[ R_1 \right]_{(n-1)}$. Therefore,
$$x=x_1+\alpha_0 y\in R_{(n)} +\alpha_0 \cdot \left[ R_1 \right]_{(n-1)}= R_{(n)}.$$
This completes the proof.
\end{proof}

We are now ready to prove Theorem \ref{main}.

\begin{proof}[Proof of Theorem \ref{main}]
Thanks to Proposition \ref{3.7}, we already have $R^*_{R_n} \subseteq  \left[ R_{(n)}\right]^*_{R_n}=R_{(n)}$. Hence it is sufficient to show the inclusion $(\supseteq)$.
Passing to the localization, we may assume that $R$ is a local ring. Moreover, enlarging the residue class field of $R$, we may assume that the residue class field of $R$ is infinite. We prove by induction on $n$.

If $n=0$, the assertion is obvious. 
Suppose $n=1$. Because the residue class field of $R$ is infinite, there exists $a_0\in R$ such that $\m_0=\overline{a_0 R}$. Here, we provide the following claim.

\begin{claim}\label{3.8}
Let $R$ be a Noetherian ring. Let $I=\overline{aR}$ with $a\in W(R)$. Then $IR^I\subseteq R^*_{R^I}$.
\end{claim}

\begin{proof}
Because $I^{n+1}=aI^n$ for some $n>0$, we can write $R^I=R\left[ \frac{I}{a} \right]= \frac{I^n}{a^n}$, so that $$IR^I=\frac{I^{n+1}}{a^n}=\frac{aI^n}{a^n}=\frac{I^n}{a^{n-1}}.$$
Let $x_1, \cdots , x_n\in I$. Then $\frac{x_1}{a}, \cdots , \frac{x_n}{a}\in R^I$, and we have $\frac{x_1x_2}{a} \in R^*_{R^I}$ since
$$\frac{x_1x_2}{a} \otimes 1= \frac{x_1}{a}\otimes x_2 = \frac{x_1}{a} \otimes a\cdot \frac{x_2}{a}=x_1\otimes \frac{x_2}{a}=1\otimes \frac{x_1x_2}{a}$$
in $R^I\otimes_R R^I$. By continuing the same operation, we get $\frac{x_1\cdots x_n}{a^{n-1}}\in R^*_{R^I}$, which induces $IR^I=\frac{I^n}{a^{n-1}}\subseteq R^*_{R^I}$.
\end{proof}

By this claim, we obtain
$$R_{(1)} =R+\m_0 R_1 \subseteq R^*_{R_1}.$$

Suppose that  $n>1$ and our assertion holds true for $n-1$. By Proposition \ref{3.4}, $\m_0\cdots \m_{n-1}$ is an integrally closed ideal in $R_{(n-1)}$. Thus, there exists $a\in R_{(n-1)}$ such that $\m_0\cdots \m_{n-1}=\overline{aR_{(n-1)}}$ ($R_{(n-1)}$ is also local and the residue class field is infinite; see Remark \ref{rem}). Moreover, because $\m_0\cdots \m_{n-1}=\alpha_0\cdots \alpha_{n-2}\m_{n-1}$, we have $(\m_0\cdots \m_{n-1})^{\ell}:(\m_0\cdots \m_{n-1})^{\ell}=\m_{n-1}^{\ell}:\m_{n-1}^{\ell}$ for every $\ell \ge 0$, so that
$$\displaystyle \left[ R_{(n-1)} \right]^{\m_0\cdots \m_{n-1}} = \bigcup_{\ell \ge0}[(\m_0\cdots \m_{n-1})^{\ell}:(\m_0\cdots \m_{n-1})^{\ell}]=\bigcup_{\ell\ge0}[\m_{n-1}^{\ell}:\m_{n-1}^{\ell}]=R_n.$$
Therefore, 
$$R_{(n)} = R_{(n-1)} + \m_0\cdots \m_{n-1} R_n \subseteq \left[ R_{(n-1)}  \right]^*_{R_n}\cdots ({\rm I})$$
by Claim \ref{3.8}. In contrast, by the hypothesis of induction, we have $R_{(n-1)}=R^*_{R_{n-1}}$. 
Hence 
$$R_{(n-1)}=R^*_{R_{n-1}}\subseteq R^*_{R_n} \cdots ({\rm I \hspace{-1pt} I}),$$
 since $R_{n-1} \subseteq R_n$. 
From (I) and (I\hspace{-1pt}I), we obtain 
$$R_{(n)} \subseteq \left[ R_{(n-1)} \right]^*_{R_n} \subseteq \left[  R^*_{R_n}   \right]^*_{R_n} =R^*_{R_n} ,$$
as desired.
\end{proof}

%%%%%%%%%%%%%%%%%%%4/20%%%%%%%%%%%%%%%%%%%%%%%%%%%%

We obtain the following corollary which we use to prove Theorem \ref{4.2}.

\begin{cor}\label{3.9}
The following conditions are equivalent. 
\begin{enumerate}[$(1)$]
\item
$R^*$ is a finitely generated $R$-module.
\item
$R^*=R_{(n)}=R+\m_0R_1+\cdots + \m_0\m_1\cdots \m_{n-1}R_n$ for some $n\ge 0$.
\item
$R_n$ is an {\it Arf ring} for some $n\ge 0$.

\end{enumerate}

\end{cor}

\begin{proof}
$(1) \Leftrightarrow (2)$ This follows from Theorem \ref{main} and the fact that $R^*$ is a Noetherian $R$-module.

$(2) \Rightarrow (3)$ Suppose that $R^*=R_{(n)}$. Then $R^*$ is also a Cohen-Macaulay ring since $R^*$ is a finitely generated $R$-module. Thus, $R^*$ is an Arf ring by Theorem \ref{2.4} (notice that $[R^*]^*=R^*$). 
In contrast, since $\left[ R_{(n)} \right]_1=\left[ R_1 \right]_{(n-1)}$,by continuing the same operation, we have
$$[R^*]_n=[R_{(n)}]_n=\left[ \left[ R_1 \right]_{(n-1)} \right]_{n-1}=\left[ \left[ R_2 \right]_{(n-2)} \right]_{n-2}= \cdots =\left[ \left[ R_n \right]_{(0)} \right]_{0}=R_n.$$
 Therefore, $R_n$ is also an Arf ring by \cite[Lemma 2.3.]{L} (or \cite[Corollary 2.5]{I}).
 
 $(3) \Rightarrow (2)$ 
Suppose that $R_n$ is an Arf ring. Then so is $R_{\ell}$ for every $\ell \ge n$ by \cite[Lemma 2.3.]{L} (or \cite[Corollary 2.5]{I}). 
Hence,  because $\m_{\ell}$ is a stable ideal in $R_{\ell}$, we have $R_{\ell+1}=\left[ R_{\ell} \right]^{\m_{\ell}} =\m_{\ell} :\m_{\ell}$ by Lemma \ref{2.1}, so that 
$$\m_{\ell}R_{\ell+1}=\m_{\ell} \subseteq R_{\ell}$$
for every $\ell \ge n$. Therefore, we have
$$R_{(\ell+1)} = R_{(\ell)} + \m_0\cdots \m_{\ell-1}\m_{\ell}R_{\ell+1}\subseteq R_{(\ell)} + \m_0\cdots \m_{\ell-1}R_{\ell}=R_{(\ell)},$$
which implies $R_{(n)}=R_{(n+1)}=\cdots =R^*$. 
\end{proof}

Concluding this section, let us give two examples.

\begin{ex}\label{3.10}
Let $V=k[[t]]$ be the formal power-series ring over a field $k$, and let $R=k[[t^3, t^{3\ell+1}]]\subseteq V$ $(\ell \ge 1)$ . Then 
$$R_{(n)}=k[[t^{3}, t^{3\ell +1}, t^{6\ell -3n +2}]]\subseteq R_n=k[[t^3, t^{3(\ell-n)+1}]]\ \ \text{for $1\le n\le \ell-1$, and}$$
$$R_{(\ell)}=R_{(\ell+1)}=\cdots =k[[t^3, t^{3\ell+1}, t^{3\ell+2}]]\subseteq R_{\ell}=\overline{R}=V.$$

Therefore, $R^*=R_{(\ell)}=R+\m_0R_1+\cdots + \m_0\m_1\cdots \m_{\ell-1}R_{\ell}=k[[t^3, t^{3\ell+1}, t^{3\ell+2}]]$.

\end{ex}

%\begin{ex}
%Let $V=k[[t]]$ be the formal power-series ring over a field $k$, and let $R=k[[t^3, t^{7}]]\subseteq V$. Then 
%$$R_{(1)}=k[[t^3, t^7, t^{11}]]\subseteq R_1=k[[t^3, t^4]],$$
%$$R_{(2)}=R_{(3)}=\cdots =k[[t^3, t^7, t^8]]\subseteq R_2=\overline{R}=V.$$

%Therefore, $R^*=R_{(2)}=R+\m_0R_1+\m_0\m_1R_2=k[[t^3, t^7, t^8]]$.

%\end{ex}

\begin{ex}\label{3.11}
Let $k[[X, Y]]$ be the 2-dimensional formal power-series ring over a field $k$, and let $R=k[[X, Y]]/(Y^3)$. We denote by $x, y$ the images of $X, Y$ in $R$.
Then
$$R_{(n)}=R\left[\frac{y^2}{x^n}\right] \subseteq R_n=R\left[\frac{y}{x^n}\right]$$
for any $n\ge 0$.
Therefore, $R^*=\underset{n\ge 0}{\bigcup}R_{(n)}=R\left[\frac{y^2}{x^n} \mid n\ge 0\right]$.
In particular, $R^*$ is neither a finitely generated $R$-module nor  a Noetherian ring.

\end{ex}

\section{finite generation of $R^*$}\label{sec4}

In this section, we characterize the finite generation of the strict closure $R^*$, when $R$ is a Cohen-Macaulay semi-local ring of dimension one with $\dim R_M=1$ for every $M\in\Max R$. We begin with the following proposition.

\begin{prop}\label{4.1}

Suppose that $R$ is a Noetherian ring and $\jcb(R)$ contains a non-zerodivisor on $R$. If $R^*$ is a finitely generated $R$-module, then $(\sqrt{(0)})^2=(0)$ in $R$.
\end{prop}

\begin{proof}
Take a non-zerodivisor $a\in \jcb(R)$ on $R$. Then, we have
$$\displaystyle \frac{(\sqrt{(0)})^2}{a^n} \subseteq \frac{(\overline{a^nR})^2}{a^n}\subseteq R^* $$
 for any $n\ge 0$, where the first inclusion follows from the fact  $\sqrt{(0)} \subseteq \overline{a^nR}$, and the second inclusion follows from the proof of Theorem \ref{main}. 
Therefore, we get the chain 
$$\displaystyle  \frac{(\sqrt{(0)})^2}{a} \subseteq  \frac{(\sqrt{(0)})^2}{a} \subseteq  \cdots \frac{(\sqrt{(0)})^2}{a^n} \subseteq  \cdots \subseteq R^*   $$ 
of $R$-modules. Because $R^*$ is a Noetherian $R$-module, there exists an integer $n\ge 1$ such that $\displaystyle \frac{(\sqrt{(0)})^2}{a^n}=\frac{(\sqrt{(0)})^2}{a^{n+1}}$, which implies $(\sqrt{(0)})^2=a(\sqrt{(0)})^2$. Consequently, we get $(\sqrt{(0)})^2=(0)$ by Nakayama's lemma, as desired. 
\end{proof}

In what follows, suppose that $R$ is a Cohen-Macaulay semi-local ring of dimension one with $\dim R_M=1$ for every $M\in\Max R$. We denote by $\widehat{R}$ the $\jcb(R)$-adic completion of $R$. It is well known that the integral closure $\overline{R}$ is a finitely generated $R$-module if and only if $\widehat{R}$ is reduced. The following is the strict closure version of  that one, which is the second main result of this paper.

\begin{thm}\label{4.2}
The following conditions are equivalent. 
\begin{enumerate}[$(1)$]
\item
$R^*$ is a finitely generated $R$-module.

\item
$(\sqrt{(0)})^2=(0)$ in $\widehat{R}$.

\item
$(P\widehat{R}_P)^2=(0)$ in $\widehat{R}_P$ for every $P\in \Min \widehat{R}$.

\end{enumerate}

\end{thm}

Before proving Theorem \ref{4.2}, we show the following.

\begin{lem}\label{4.3}

$[ \widehat{R} ]^*=R^*\otimes_R \widehat{R}$.

\end{lem}

\begin{proof}

Notice that we can consider $\rmQ(R) \otimes_R \widehat{R} = \rmQ(\widehat{R})$. Because $\jcb(\widehat{R})=\underset{M\in \Max R}{\bigcap} M\widehat{R} =\jcb(R) \widehat{R},$ for large some $\ell \ge 1$, we have
$$[\widehat{R}]_1 =\jcb(R)^{\ell}\widehat{R}:\jcb(R)^{\ell}\widehat{R}=[\jcb(R)^{\ell}:\jcb(R)^{\ell}]\otimes_R \widehat{R}=R_1\otimes_R \widehat{R} =\widehat{R_1}.$$ 
By continuing the same operation, we get
$$[\widehat{R}]_n=R_n\otimes_{R_{n-1}} \widehat{R_{n-1}} =\widehat{R_n}=R_n\otimes_{R} \widehat{R}\ \  \text{and \  } \jcb([\widehat{R}]_n)=\jcb(R_n)\widehat{R},$$
so that
 $$[\widehat{R}]_{(n)}=\widehat{R}+\jcb(\widehat{R})[\widehat{R}]_1+\cdots +\jcb(\widehat{R})\cdots \jcb([\widehat{R}]_{n-1})[\widehat{R}]_n=R_{(n)}\otimes_R \widehat{R}$$
for every $n\ge 1$. Therefore, we get
$$\displaystyle [\widehat{R}]^*=\underset{n\ge 0}{\bigcup} [\widehat{R}]_{(n)}=\lim_{n\to \infty}[R_{(n)}\otimes_R \widehat{R}]=[\lim_{n\to \infty}R_{(n)}]\otimes_R \widehat{R}=R^*\otimes_R \widehat{R}$$
by Theorem \ref{main}, as desired. 
\end{proof}

We start the proof of Theorem \ref{4.2}.

\begin{proof}[Proof of Theorem \ref{4.2}]

  Notice that $$\rmQ(\widehat{R})=\underset{P\in \Min \widehat{R}}{\prod} [\widehat{R}]_P \text{\ \ and\ \ } \sqrt{(0)}\rmQ(\widehat{R}) = \underset{P\in \Min \widehat{R}}{\prod} P[\widehat{R}]_P,$$  since $\Ass \widehat{R}= \Min \widehat{R}$. 
  Therefore, the equivalence $(2) \Leftrightarrow (3)$ holds since $(\sqrt{(0)})^2 =(0)$ in $\widehat{R}$ if and only if $(\sqrt{(0)}\rmQ(\widehat{R}))^2 =(0)$ in $\rmQ(\widehat{R})$.
  
  $(1) \Rightarrow  (2)$ Suppose that $R^*$ is a finitely generated $R$-module. Since $[\widehat{R}]^*=R^*\otimes_R \widehat{R}$ by Lemma \ref{4.3}, $[\widehat{R}]^*$ is also a finitely generated $\widehat{R}$-module, so that $(\sqrt{(0)})^2=(0)$ in $\widehat{R}$ by Proposition \ref{4.1}.

To prove the implication $(2) \Rightarrow (1)$, we need some preparation.   Let $B=R/\sqrt{(0)}$. Then, $B$ is also a Cohen-Macaulay semi-local ring with $\dim B_N=1$ for every $N\in \Max B$, and reduced. Passing to the natural surjective map $\rmQ(R)\to \rmQ(B)$, we consider $\rmQ(B)=\rmQ(R)/\sqrt{(0)}\rmQ(R)$. Moreover, the following lemma holds true.

\begin{lem}\label{4.4}
The following assertions hold true.

\begin{enumerate}[$(1)$]

\item
$\overline{B}=\overline{R}/\sqrt{(0)}\rmQ(R)$.

\item
Suppose that $R=\widehat{R}$. Then $B_n=R_n/\sqrt{(0)R_n} = [R_n+\sqrt{(0)}\rmQ(R)]/\sqrt{(0)}\rmQ(R)$ for any $n\ge 0$.

\end{enumerate}

\end{lem}

\begin{proof}
$(1)$  $\overline{R}/\sqrt{(0)}\rmQ(R)\subseteq \overline{B}$ is clear. Let $\varphi \in \overline{B}$, and write $\varphi = \overline{(\frac{a}{s})} $ $(s\in W(R), a\in R)$, where $\overline{*}$ denotes the image of $*$ in $\rmQ(B)=\rmQ(R)/\sqrt{(0)}\rmQ(R)$. Then,
$$\left(\frac{a}{s} \right)^n + a_1\left(\frac{a}{s} \right)^{n-1} + \cdots +  a_{n-1}\left(\frac{a}{s} \right) + a_n \in \sqrt{(0)} \rmQ(R)$$ 
for some $n\ge 1$ and $a_i \in R$ $(1\le i \le n)$. Therefore,
$$\left( \left(\frac{a}{s} \right)^n + a_1\left(\frac{a}{s} \right)^{n-1} + \cdots +  a_{n-1}\left(\frac{a}{s} \right) + a_n \right)^m=0 \text{\ \  in  $\rmQ(R)$}$$
for some $m\ge 1$, which implies $\frac{a}{s}\in \overline{R}$.

$(2)$  We prove by induction on $n$. If $n=0$, then
$$B_0=B=R_0/\sqrt{(0)}=[R_0+\sqrt{(0)}\rmQ(R)]/\sqrt{(0)}\rmQ(R).$$
Suppose that $n>0$ and our assertion holds true for $n-1$. Then,
$$B_{n-1}=R_{n-1}/\sqrt{(0)R_{n-1}}=[R_{n-1}+\sqrt{(0)}\rmQ(R)]/\sqrt{(0)}\rmQ(R),  \text{\  and $\jcb(B_{n-1})=\jcb(R_{n-1})B_{n-1}$}.$$
In contrast, since $B$ is a complete reduced ring, $\overline{B_{n-1}}=\overline{B}$ is a direct product of discrete valuation rings.  Therefore, we can write   
$$\jcb(B_{n-1})=\jcb(R_{n-1})B_{n-1}=\overline{\jcb(R_{n-1})\overline{B}}\cap B_{n-1}=\jcb(R_{n-1})\overline{B}\cap B_{n-1}=\overline{a} \overline{B}\cap B_{n-1}=\overline{\overline{a}B_{n-1}}$$
for some $a\in \jcb(R_{n-1})$, where the second equality follows from the fact that $\jcb(B_{n-1})$ is an integrally closed ideal in $B_{n-1}$. Moreover, because $\overline{\overline{a}B_{n-1}}=(\overline{aR})B_{n-1}$ (see {\cite[Proposition 1.1.5]{SH}}), we get $\jcb(R_{n-1})=\overline{aR_{n-1}}$. Hence we have 
$$R_n=[R_{n-1}]^{\jcb(R_{n-1})}=R_{n-1}\left[\frac{\jcb(R_{n-1})}{a}\right], \text{\ and \ } B_n= [B_{n-1}]^{\jcb(R_{n-1})B_{n-1}}=B_{n-1}\left[\frac{\jcb(R_{n-1})B_{n-1}}{\overline{a}}\right],$$
which implies that $B_{n}=R_{n}/\sqrt{(0)R_{n}}=[R_{n}+\sqrt{(0)}\rmQ(R)]/\sqrt{(0)}\rmQ(R)$, as desired.
\end{proof}

We are now ready to prove the implication $(2) \Rightarrow (1)$.  Suppose that $(\sqrt{(0)})^2=(0)$ in $\widehat{R}$. Since $[\widehat{R}]^*=R^*\otimes_R \widehat{R}$, we may assume that $R=\widehat{R}$. Thus, since $\overline{B}$ is a finitely generated $B$-module, there exists an integer $n\ge 0$ such that $B_{\ell} = \overline{B}$ for every $\ell \ge n$, so that
$$\overline{R}=R_{\ell}+\sqrt{(0)}\rmQ(R)$$
for any $\ell \ge n$ by Lemma \ref{4.4}. According to Corollary \ref{3.9}, it suffices to show that $R_n$ is an Arf ring. Let $I\in \calX_{R_{n}}$. By the proof of Lemma \ref{4.4} $(2)$, there exists $a\in I$ such that $I=\overline{aR_{n}}$. Therefore, we have
$$ I=a\overline{R}\cap R_{n}=\left[a\left(R_{n}+\sqrt{(0)}\rmQ(R)\right)\right]\cap R_{n}=[aR_{n} + \sqrt{(0)}\rmQ(R)]\cap R_{n}=aR_{n}+\sqrt{(0)R_{n}},$$
so that $$I^2=aI + a\sqrt{(0)R_{n}}+(\sqrt{(0)R_{n}})^2=aI$$ (notice that $R_{n}$ also satisfies the condition $(\sqrt{(0)R_{n}})^2=(0)$). Thus, $I$ is stable, so that $R_n$ is an Arf ring by Theorem \ref{2.3}. This completes the proof of Theorem \ref{4.2}.
\end{proof}

\begin{rem}\label{4.5}
Even if the condition $(\sqrt{(0)})^2=(0)$ is satisfied in $R$, $R^*$ is not necessarily a finitely generated $R$-module.  In fact, let $A=k[[X, Y]]/(Y^3)$. Then $(\sqrt{(0)})^2\neq (0)$ in $A$.  However, it is known that such a ring $A$ is a completion of some Noetherian local domain $R$ (see, e.g., \cite[Theorem 1]{Le}). 
%We set $C=A\oplus A$ and define the multiplication on $C$ by
%$$(a, x)\cdot (b, y)=(ab, ay+bx)$$ for $(a, x), (b, y)\in C$. Then, $C=A\ltimes A$ is a Cohen-Macaulay local ring with $\dim C=\dim A =1$, called $\it{ the\  idealization\  A\  over\  A }$. Moreover, since $\sqrt{(0)C}=\sqrt{(0)A}\times A$, $(\sqrt{(0)C})^2=(\sqrt{(0)A})^2\times \sqrt{(0)A} \neq (0)$ in $C$. 

\end{rem}

Concluding this paper, let us give a few consequences of Theorem \ref{4.2}. 

First, we can weaken the condition of assumptions in \cite[Corollary 7.8]{C}. We need not assume the finite generation of $\overline{R}$.

\begin{cor}\label{4.6}
Suppose that $(\sqrt{(0)})^2=(0)$ in $\widehat{R}$. Then $R^*$ coincides with the Arf closure of $R$, which is the smallest Arf ring between $R$ and $\overline{R}$.
\end{cor}

\begin{proof}
Since $R^*$ is a finitely generated $R$-module, $R^*$ is a Noetherian ring with $[R^*]^*=R^*$. Therefore,  $R^*$ is  an Arf ring by Theorem \ref{2.4}. In contrast,  let $A$ be an Arf ring between $R$ and $\overline{R}$. Then, we have $R^*\subseteq A^*=A$ by Theorem \ref{2.4}.
\end{proof}

Second, we have the following.

\begin{cor}\label{4.7}
Suppose that $(\sqrt{(0)})^2=(0)$ in $\widehat{R}$ and $\widehat{R}/\sqrt{(0)}$ is integrally closed. Then $R=R^*$.

\end{cor}

\begin{proof}
We may assume that $R=R^*$. By the proof of $(2)\Rightarrow (1)$ in Theorem \ref{4.2}, $R_0=R$ is an Arf ring since $B=R/\sqrt{(0)}$ is integrally closed. Therefore, $R=R^*$ by Theorem \ref{2.4}, as desired. 
\end{proof}

\begin{rem}\label{4.8}
\begin{enumerate}[$(1)$]
\item
The converse of Corollary \ref{4.7} does not hold in general. For example, let $R=k[[X, Y]]/(XY)$. Then $R=\widehat{R}$, $R=R^*$, and $R$ is reduced, but $R\neq \overline{R}$. 

\item
In Corollary \ref{4.7}, we cannot remove the assumption that  $(\sqrt{(0)})^2=(0)$ in $\widehat{R}$. In fact, let $R=k[[X, Y]]/(Y^3)$. Then $R=\widehat{R}$, and $R/\sqrt{(0)}\cong k[[X]]$ is integrally closed. However, $R\neq R^*$ by Example \ref{3.11}.

\end{enumerate}

\end{rem}

\end{document}